\documentclass[a4paper,12pt,reqno,twoside]{amsart}

\usepackage[utf8]{inputenc} 		
\usepackage[T1]{fontenc} 

\usepackage{amssymb}
\usepackage{amsmath}
\usepackage{graphicx}
\usepackage[colorlinks,citecolor=blue,urlcolor=blue]{hyperref}
\usepackage{cite}
\usepackage[hmargin=2cm,vmargin=2.5cm]{geometry}
\usepackage{mathrsfs} 

\newcommand{\be}{\begin{eqnarray}}
\newcommand{\ee}{\end{eqnarray}}

\newcommand{\beq}{\begin{equation}}
\newcommand{\eeq}{\end{equation}}
\newcommand{\beqn}{\begin{equation*}}
\newcommand{\eeqn}{\end{equation*}}

\newcommand{\imag}{\mathrm{i}}

\newcommand{\round}[1]{\lfloor#1\rfloor}

\newtheorem{thm}{Theorem}[section]

\newtheorem{lem}[thm]{Lemma}
\newtheorem{defn}[thm]{Definition}

\newtheorem{example}[thm]{Example}

\newcommand\cH{{\mathcal H}}

\newcommand\cN{{\mathcal N}}

\newcommand\cQ{{\mathcal Q}}

\newcommand\bN{{\mathbb N}}

\newcommand\bR{{\mathbb R}}

\newcommand\bT{{\mathbb T}}

\newcommand\bZ{{\mathbb Z}}

\newcommand\rd{{\mathrm d}}

\newcommand{\ve}{\varepsilon}

\def\bfP{\mathbf{P}}

\def\bfx{\mathbf{x}}
\def\bfy{\mathbf{y}}

\begin{document}

\title[Finite-dimensional distributions of dispersing billiard processes]{A note on the finite-dimensional distributions of dispersing billiard processes}

\author[Juho Lepp\"anen]{Juho Lepp\"anen}
\address[Juho Lepp\"anen]{
Department of Mathematics and Statistics, P.O.\ Box 68, Fin-00014 University of Helsinki, Finland.}
\email{juho.leppanen@helsinki.fi}

\author[Mikko Stenlund]{Mikko Stenlund}
\address[Mikko Stenlund]{
Department of Mathematics and Statistics, P.O.\ Box 35, Fin-40014 University of Jyv\"askyl\"a, Finland.}
\email{mikko.s.stenlund@jyu.fi}


 


\begin{abstract}
In this short note we consider the finite-dimensional distributions of sets of states generated by dispersing billiards with a random initial condition. We establish a functional correlation bound on the distance between the finite-dimensional distributions and corresponding product distributions. We demonstrate the usefulness of the bound by showing that it implies several limit theorems. The purpose of this note is to provide a tool facilitating the study of more general functionals of the billiard process.
\end{abstract}

\maketitle


\subsection*{Acknowledgements}
This work was supported by the Jane and Aatos Erkko Foundation, and by Emil Aaltosen S\"a\"ati\"o. 


\section{Introduction}

%


In this note we revisit the two-dimensional dispersing Sinai billiard with finite horizon. To specify the model, we consider the torus $\bT^2$ with a finite collection of scatterers, i.e., closed convex sets $S_1,\dots,S_M$, having $C^3$ boundaries $\partial S_m$ with strictly positive curvatures. A particle moves linearly in the domain $\bT^2\setminus\cup_{m=1}^M S_m$, with unit speed, up to elastic collisions with the boundaries $\partial S_m$ of the scatterers. The scatterers are disjoint and positioned so that the free path length of the particle is bounded. Note that the precise dynamics of the system is fully determined by the geometry of the domain $\bT^2\setminus\cup_{m=1}^M S_m$. 

A standard discrete-time representation of the dynamics is obtained by keeping track of the collisions only, which leads to the so-called collision map $T:X\to X$ as follows: Topologically, $X$ is the disjoint union of $m$ cylinders $X_m$, homeomorphic to $\partial S_m\times[-\frac\pi2,\frac\pi2]$. A general point $x\in X$ consists of a pair $x = (r,\varphi)$, where $r$ represents the position of the particle on the boundary $\cup_{m=1}^M \partial S_m$ during a collision, and $\varphi$ represents its direction immediately after the collision relative to the normal of the boundary. Then $Tx = (r',\varphi')$ is defined as the corresponding pair after the next collision. Since the continuous-time system is Hamiltonian, preserving phase-space volume, the collision map preserves a corresponding Borel probability measure, namely $\rd\mu(r,\varphi) = \text{const}\cdot\cos\varphi\,\rd r\,\rd\varphi$, on $X$. Reversing the velocity of the particle, one moreover verifies that $T$ is invertible.

Given an initial state $x\in X$, the billiard dynamics generates the sequence of states $(T^ix)_{i\in\bN}\in X^\bN$. If $x$ is chosen at random, according to the invariant measure~$\mu$, the sequence is a stationary random process, which we call the billiard process. We can equally well define the two-sided billiard process $(T^ix)_{i\in\bZ}$, and everything below extends readily to that setup, but let us proceed with the one-sided case. Of course, knowledge of the value of $T^ix$ for some $i$ fully determines the value of $T^jx$ for all $j$. Yet, the same is in general not true of, say, $f(T^ix)$ and $g(T^jx)$ where $f,g:X\to\bR$ are ``observables''. For instance, if~$f$ and $g$ are H\"older continuous, then an exponential covariance bound
\beq\label{eq:pair}
\left|\int f(T^ix)g(T^jx)\,\rd\mu(x) - \int f(T^ix)\,\rd\mu(x)\int g(T^j x)\,\rd\mu(x)\right| \le C\theta^{j-i}
\eeq
holds. Above, $C>0$ and $\theta\in(0,1)$ depend on the H\"older classes of $f,g$ and on system constants\footnote{In this paper system constants are quantities which only depend on the geometry of the domain $\bT^2\setminus\cup_{m=1}^M S_m$.}. Obviously, more general bounds exist, but~\eqref{eq:pair} is sufficient for the ongoing illustrative discussion.
The covariance bound~\eqref{eq:pair}, and other similar results, are consequences of the chaotic nature of the billiard dynamics. Colloquially, we may regard the observations $f(T^ix)$ as weakly dependent random variables, which in one form or another is at the heart of proving probabilistic limit results for functionals $F\circ (T^ix)_{i\in\bN}$ of the billiard process, say concerning the asymptotic behavior of the Birkhoff sums
\beqn
S_N(x) = \sum_{i=0}^{N-1} f(T^i x)
\eeqn
in the limit $N\to\infty$.

To proceed, we recall that a finite-dimensional distribution of the billiard process is the joint distribution~$\bfP_I$ of a subsequence~$(T^i(x))_{i\in I}$ corresponding to a finite index set $I = (i_1,i_2\dots,i_p)\subset\bN$. From here on, we will without loss of generality always assume that the indices in such an index set are in increasing order, $i_1<\dots<i_p$. The probability measure~$\bfP_I$ on~$X^p$ is characterized by the identity
\beqn
\int h \, \rd\bfP_I = \int h(T^{i_1}x,T^{i_2}x,\dots,T^{i_p}x)\,\rd\mu(x)
\eeqn
for bounded measurable functions $h:X^p\to\bR$. Of course, stationarity of the billiard process means that $\bfP_I = \bfP_{I'}$ for all translates $I' = (i_1+k,\dots,i_p+k)$ of $I$, which is clearly true by the invariance of~$\mu$.

For example, in terms of finite-dimensional distributions,~\eqref{eq:pair} reads
\beqn
\left|\int f\otimes g\,\rd\bfP_{(i,j)} - \int f\otimes g\,\rd(\bfP_i\otimes\bfP_j)\right| \le C\theta^{j-i}
\eeqn
where $(f\otimes g)(x,y) = f(x)g(y)$. In this weak sense, we may informally write
\beqn
\text{``\ $\bfP_{(i,j)}\approx\bfP_i\otimes\bfP_j$\ ''}
\eeqn
when $j-i$ is large.

\medskip
\noindent{\bf Convention.} In the rest of the note we will consider unions
\beqn
I = \bigcup_{1\le k\le K} I_k
\eeqn
of increasing nonempty index sets $I_k = (i_{p_{k-1}+1},\dots,i_{p_k})\subset\bN$, where $p_0=0$. We will always assume they are disjoint and ordered, in the sense that the gap between~$I_k$ and~$I_{k+1}$ satisfies
\beqn
\ell_k = i_{p_k+1} - i_{p_k} > 0
\eeqn 
for all $k=1,\dots,K-1$. We shall henceforth write
\beqn
I_1\le \dots \le I_K
\eeqn
to express these conventions succinctly.

\medskip

Being still informal, higher order correlation bounds indicate that, when each $\ell_k$ is large,
\beq\label{eq:approx}
\text{``\ $\bfP_I \approx \bfP_{I_1}\otimes\dots\otimes\bfP_{I_K}$\ ''}
\eeq
in some weak sense. 
The purpose of this brief note is to make this interpretation precise. As an aside, it provides a unified perspective on several limit theorems that we treat as examples: We will obtain an estimate on the difference between $\bfP_I$ and $\bfP_{I_1}\otimes\dots\otimes\bfP_{I_K}$ in an appropriately general sense of practical use, which is then shown to imply all the limit theorems. Let us immediately be clear that the latter limit theorems, per se, have been proved elsewhere, in the references cited (although we do obtain some minor improvements). Thus a side goal here is to shed additional light on why those theorems are true, in a mathematically rigorous way. The main result is the aforementioned estimate, which we call the ``functional correlation bound''. We expect it to be of much broader use, as it is directly applicable to studying other kinds of functionals of the billiard process than the examples included here. In short, we view the functional correlation bound as a tool which helps put the vague statement in~\eqref{eq:approx} onto a solid footing, in reasonable generality, so that it can be used effectively in technical proofs in the theory of dispersing billiards.

\medskip
\noindent{\bf Structure of this note.} In Section~\ref{sec:results} we state two theorems on functional correlation decay. In Section~\ref{sec:examples} we give examples of using them for deducing limit theorems, and in Section~\ref{sec:proofs} we prove them.

\section{Results}\label{sec:results}
Before stating the results, we recall a few standard facts from the theory of dispersing billiards. The reader is referred to the book~\cite{ChernovMarkarian_2006} for more details.

In the disjoint union~$X=\cup_{m=1}^M X_m$, the cylinders $X_m = \partial S_m\times[-\frac\pi2,\frac\pi2]$ are further divided into horizontal strips $X_{m,k} = \partial S_m\times\{b_k<\varphi<b_{k+1}\}$, $k\in\bZ$, called homogeneity strips. (Here the numbers $b_k$ are such that $b_k-b_{k+1} = O(k^{-3})$, which facilitates controlling distortions of the map $T$ within each strip.) We consider the totality of the homogeneity strips the connected components of $X$. For a pair of points~$x,y\in X$, we say that their trajectories separate when $T^n x$ and $T^n y$ are in different components for the first time $n\ge 0$; this $n$ is called the future separation time, which we denote by~$s_+(x,y)$. We define $s_+(x,y) = \infty$ if the trajectories never separate. The past separation time~$s_-(x,y)$ is the analogous notion for the inverse map~$T^{-1}$.

A local stable manifold $W^s(x)$ of a point $x\in X$ is a maximal~$C^2$ curve such that~$T^n W^s(x)$ is completely contained in a component of~$X$, for all~$n\ge 0$. That is, given $n\ge 0$, there exist $m,k$ such that $T^nW^s(x)\subset X_{m,k}$. It can be shown that almost every point has a nontrivial local stable manifold, and that the length of~$T^n W^s(x)$ decreases exponentially as~$n\to\infty$. Given two points~$x,y\in X$, we either have~$W^s(x)=W^s(y)$ (meaning $y\in W^s(x)$) or~$W^s(x)\cap W^s(y) = \emptyset$. Note that, in the first case, $s_-(T^n x,T^n y) = n + s_-(x,y)$ for all $n\ge 0$. The family of all local stable manifolds is uncountable, and forms a measurable partition of~$X$. Local unstable manifolds $W^u(x)$ have identical properties in terms of the inverse map~$T^{-1}$. In particular, they form a measurable partition of~$X$. Moreover, if $y\in W^u(x)$, then $s_+(T^{-n}x,T^{-n}y) = n + s_+(x,y)$ for all $n\ge 0$.

We also recall the notion of dynamical H\"older continuity. The following definition is from \cite{Stenlund_2010}. It is a small variation of the one in \cite{Chernov_2006}, but in the current form it enjoys the property of being dynamically closed, which is used in the proofs; see Lemma~\ref{lem:dynHolder}.
\begin{defn}
A function~$g:X\to\bR$ is dynamically H\"older continuous on local unstable manifolds with rate $\vartheta\in(0,1)$ and constant $c\ge 0$ if
\beqn
|g(x) - g(y)| \le c\,\vartheta^{s_+(x,y)}
\eeqn
holds whenever $x$ and $y$ belong to the same local unstable manifold. In this case we write $g\in\cH_+(c,\vartheta)$. Likewise, $g$ is dynamically H\"older continuous on local stable manifolds if
\beqn
|g(x) - g(y)| \le c\,\vartheta^{s_-(x,y)}
\eeqn
holds whenever $x$ and $y$ belong to the same local stable manifold. In this case we write $g\in\cH_-(c,\vartheta)$.
\end{defn}
For instance, if $g:X\to\bR$ is H\"older continuous with exponent $\alpha\in(0,1)$ and constant $|g|_\alpha$, then $g\in\cH_-(c,\vartheta)\cap \cH_+(c,\vartheta)$, where $c = \text{const}\cdot |g|_\alpha$, and $\vartheta = \vartheta(\alpha)$ is determined by~$\alpha$ and system constants.

Finally, we introduce the class of admissible test functions $F$:
\begin{defn}\label{defn:admissible}
Given increasingly ordered index sets $I_1\le \dots\le I_K$, $K\ge 2$, we say that a bounded function $F:X^{p_K}\to\bR$ is $(I_1,\dots,I_K)$-admissible, if it is separately dynamically H\"older continuous in the sense that 
\beqn
x_r\mapsto F(x_1,\dots,x_{p_K})
\in
\begin{cases}
\cH_+(c_r,\vartheta_+), & 1 \le r \le p_1,
\\
\cH_+(c_{r+},\vartheta_+) \cap \cH_-(c_{r-},\vartheta_-), & p_1+1 \le r\le p_{K-1},
\\
\cH_-(c_r,\vartheta_-), & p_{K-1}+1 \le r\le p_K.
\end{cases}
\eeqn
\end{defn}

Here is our first functional correlation bound, concerning the case $K=2$:
\begin{thm}\label{thm:pair}
There exist system constants $M_0,M_1>0$ and $\theta_0,\theta_1\in(0,1)$ such that the following holds. Let $I_1\le I_2$, and let $F$ be $(I_1,I_2)$-admissible:
\beqn
x_r\mapsto F(x_1,\dots,x_{p_2})
\in
\begin{cases}
\cH_+(c_r,\vartheta_+), & 1\le r \le p_1,
\\
\cH_-(c_r,\vartheta_-), & p_{1}+1 \le r\le p_2.
\end{cases}
\eeqn
Then
\beqn
\begin{split}
& \left| \int F\,\rd\bfP_{I_1\cup I_2} -  \int F\,\rd(\bfP_{I_1}\otimes\bfP_{I_2}) \right|
\\
\le\ & \left(\sum_{r = 1}^{p_1} c_r \vartheta_+^{i_{p_1} - i_r}\right)\!\vartheta_+^{\frac14 \ell_1} 
+ M_0\!\left(\,\sum_{r=p_1+1}^{p_2} c_r\vartheta_-^{i_r - i_{p_1+1}} + \|F\|_\infty\right)\! \max(\theta_0,\vartheta_-)^{\frac14 \ell_1-\frac13} 
\\
&\qquad + 2M_1\|F\|_\infty  \theta_1^{\frac14 \ell_1-\frac13}.
\end{split}
\eeqn
Here $\ell_1 = i_{p_1+1}-i_{p_1}$ is the gap between $I_1$ and $I_2$. 
\end{thm}

The second functional correlation bound extends the first one to $K\ge 2$. While it is entirely possible to formulate the result for general admissible test functions~$F$, the resulting bound has a cumbersome expression. For aesthetic reasons alone, we restrict to functions admissible with the same parameters, leaving generalizations to the reader. By ``the same parameters'' we mean that $c_r\equiv c_{r\pm}\equiv c$ and $\vartheta_-=\vartheta_+$ in Definition~\ref{defn:admissible}.

\begin{thm}\label{thm:multiple}
Let $I_1\le\dots\le I_K$, $K\ge 2$, and let $F$ be $(I_1,\dots,I_K)$-admissible, with the same parameters $c\ge 0$ and $\vartheta\in(0,1)$. Then 
\beqn
\left|\int F\,\rd\bfP_{I} - \int F\,\rd(\bfP_{I_1}\otimes\dots\otimes\bfP_{I_K})\right| \le M\!\left(\frac{c}{1-\vartheta} + \|F\|_\infty \right)\!\sum_{k=1}^{K-1}\theta^{\ell_k}. 
\eeqn
Here $\ell_k = i_{p_k+1}-i_{p_k}$ is the gap between $I_k$ and $I_{k+1}$,
\beqn
\theta = \max(\vartheta,\theta_0,\theta_1)^{\frac14}
\eeqn
and
\beqn
M = \max\Bigl(1+M_0\theta_0^{-\frac13}\, ,\,M_0\theta_0^{-\frac13}+2M_1\theta_1^{-\frac13}\Bigr).
\eeqn
The system constants $M_0,M_1$ and $\theta_0,\theta_1$ are the same as in Theorem~\ref{thm:pair}.
\end{thm}

A result in the spirit of Theorem~\ref{thm:multiple} was recently proved by Lepp\"anen~\cite{Leppanen_2017}, for a class of one-dimensional, non-uniformly expanding dynamical systems.

In fact, the inductive proof of Theorem~\ref{thm:multiple} shows that the special case $K = 2$ and the general case $K\ge 2$ are equivalent. This hinges on the dynamical closedness of the function classes $\cH_-$ and $\cH_+$; see Lemma~\ref{lem:dynHolder}. 

At first, the theorems may seem like inconsequential extensions of correlation bounds such as the one displayed in~\eqref{eq:pair}. But they do allow for estimating integrals of functionals of the billiards process, $\int F\circ (T^ix)_{i\in\bN}\,\rd\mu$, beyond the scope of simple correlation bounds. Just to give a simple example, consider a situation of the following kind:

\begin{example}\label{exa:sums}
Let $A:\bR^K\to\bR$ be Lipschitz continuous with constant $L$ in each variable, and let the index sets $I_1\le \dots\le I_K$ be as above. Define the sums
\beqn
S^{(k)} = \sum_{r=p_{k-1}+1}^{p_k} f_r\circ T^{i_r}, \qquad 1\le k\le K,
\eeqn
where the functions $f_r:X\to\bR$ are bounded, and dynamically H\"older continuous with the same parameters~$c\ge 0$ and $\vartheta\in(0,1)$.
Let us consider the intergral
\beqn
\mathscr{I} = \int A\bigl(S^{(1)}(x) \, ,\dots,\, S^{(K)}(x)\bigr)\,\rd\mu(x). 
\eeqn
We would like to argue that, when each $\ell_k = i_{p_k+1} - i_{p_k}$ is large, the sums in the argument of $A$ are weakly dependent, so $\mathscr{I}$ must be close to
\beqn
\mathscr{I}^\otimes =  \int A\!\left(S^{(1)}(x_1),\dots,S^{(K)}(x_K)\right) \rd\mu^{\otimes K}(x_1,\dots,x_K),
\eeqn
where the sums are literally independent due to the product measure. 
Theorem~\ref{thm:multiple} helps make such an argument rigorous: Let $F:X^{p_K}\to\bR$ be the function
\beqn
F(x_1,\dots,x_{p_K}) = A\!\left(\sum_{r=1}^{p_1} f_r(x_r)\, ,\dots, \, \sum_{r={p_{K-1}+1}}^{p_K} f_r(x_r)\right).
\eeqn
Then $F$ is $(I_1,\dots,I_K)$-admissible:
$
x_r \mapsto F(x_1,\dots,x_{p_K}) \in \cH_+(Lc ,\vartheta)\cap\cH_-(Lc,\vartheta)
$
for all $r$, and 
\beqn
\|F\|_\infty \le C = 
\begin{cases}
\|A\|_\infty & \text{if $A$ is bounded},
\\
L\sum_{r=1}^{p_K}\|f_r\|_\infty + A(0,\dots,0) & \text{if $A$ is unbounded}.
\end{cases}
\eeqn
In both cases, we immediately arrive at the quantitative estimate
\beqn
|\mathscr{I} - \mathscr{I}^\otimes| \le M\!\left(\frac{Lc}{1-\vartheta} + C \right)\!\sum_{k=1}^{K-1}\theta^{\ell_k}.
\eeqn
\end{example}

For instance, Example~\ref{exa:sums} applies to ``interlaced'' covariances of the form
\beqn
\begin{split}
& \int A_1(S^{(1)} + S^{(3)} + \dots + S^{(K-1)})\,A_2(S^{(2)} + S^{(4)} + \dots + S^{(K)})\,\rd\mu
\\
-\ & \int A_1(S^{(1)} + S^{(3)} + \dots + S^{(K-1)})\,\rd \mu \int A_2(S^{(2)} + S^{(4)} + \dots + S^{(K)})\,\rd\mu,
\end{split}
\eeqn
where the argument of the Lipschitz function $A_1$ (respectively, $A_2$) involves $f_r\circ T^{i_r}$ with $i_r\in I_k$ and $k$ odd (respectively, $k$ even). Here $K$ is even for convenience. This is so, because both terms in the difference can be compared with the integral with respect to the product measure~$\rd\mu^{\otimes K}$.

In the special case of singleton index sets, $I_k = \{i_k\}$, Example~\ref{exa:sums} yields a bound on
\beqn
\begin{split}
&\int A(f_1(T^{i_1}x),\dots, f_K(T^{i_{K}}x)) \,\rd\mu(x) 
- \int A(f_1(T^{i_1}x_1),\dots, f_K(T^{i_{K}}x_K))\,\rd\mu^{\otimes K}(x_1,\dots,x_K).
\end{split}
\eeqn
Such bounds are relevant, e.g., for multiple recurrence problems.


\section{More examples}\label{sec:examples}
In this section we give applications of Theorem~\ref{thm:multiple} to limit results. We reiterate that the sole purpose of this section is to illustrate the usefulness of the theorem: it allows to check the conditions of various limit theorems with great ease. The verified conditions actually amount to \emph{very} simple special cases of Theorem~\ref{thm:multiple}. We thus believe the result to be a tool of much broader use in analyzing dispersing billiard dynamics.

Below, the various constants in the results concerning billiards will be the same as in Theorem~\ref{thm:multiple}.

\subsection{Multiple correlation bounds}

\begin{thm}\label{thm:multicorr_bound}
Let $f_0,\dots, f_r\in \cH_+(c,\vartheta)$ and $g_0,\dots,g_k\in \cH_-(c,\vartheta)$, and define
\beqn
\tilde f = f_0 \cdot f_1\circ T^{-1} \cdots f_r\circ T^{-r}
\quad\text{and}\quad
\tilde g = g_0\cdot g_1\circ T^{1} \cdots g_k\circ T^{k}.
\eeqn
Suppose \(\| f_u \|_{\infty} = \max_{0 \le i \le r} \| f_i \|_{\infty} \) and \(\| g_v \|_{\infty} = \max_{0 \le i \le k} \| g_i \|_{\infty} \). Then
\beqn
\left | \int \tilde f \cdot  \tilde g\circ T^{n}\,\rd\mu - \int \tilde f\,\rd\mu \int\tilde g\,\rd\mu \right | \leq C\theta^n
\eeqn
for all $n\geq 0$, where
\begin{align*}
C = M\Vert f_u \Vert_{\infty}^r\Vert g_v \Vert_{\infty}^{k}\max\{\Vert f_u \Vert_{\infty}, \Vert g_v \Vert_{\infty} \}\left(\frac{c}{1-\vartheta} +  \min\{\Vert f_u \Vert_{\infty}, \Vert g_v \Vert_{\infty} \}\right).
\end{align*}
\end{thm}

It was shown in \cite{Chernov_2006} that such a multiple correlation bound suffices for the central limit theorem to hold: If $f\in\cH_-(c,\vartheta)\cap\cH_+(c,\vartheta)$ is bounded and $\int f\,\rd\mu = 0$, then
\beqn
\frac{1}{\sqrt N} \sum_{i = 0}^{N-1} f\circ T^i
\eeqn
converges weakly, as $N\to\infty$, to the normal distribution $\cN(0,\sigma^2)$ with zero mean and variance
\beqn
\sigma^2 = \int f^2\,\rd\mu + 2\sum_{i = 1}^\infty \int f\,f\circ T^i\,\rd\mu.
\eeqn
See also~\cite{Pene_2004} for a closely related result.
To be technically accurate, \cite{Chernov_2006} dealt with a smaller class of observables, as did~\cite{Pene_2004}. In \cite{Stenlund_2010} it was shown that, for the present classes~$\cH_\pm$, the multiple correlation bound is equivalent to the pair correlation bound corresponding to the special case $r=k=0$; and consequently that the pair correlation bound alone is enough for the central limit theorem.

\begin{proof}[Proof of Theorem~\ref{thm:multicorr_bound}] Define \(I_1 = (0,\ldots, r)\) and \(I_2 = (n+r,\ldots, n + r + k) \). Then
\begin{align*}
\int \tilde f \cdot  \tilde g\circ T^{n}\,\rd\mu = \int F \, \rd\bfP_{I_1 \cup I_2}
\quad\text{and}\quad
 \int \tilde f\,\rd\mu \int\tilde g\,\rd\mu = \int F \, \rd(\bfP_{I_1}\otimes\bfP_{I_2})
\end{align*}
where \(F : \, X^{r+k+2} \to \bR\) is the function
\begin{align*}
F(x_1,\ldots, x_{r+1},x_{r+2},\ldots, x_{r+k+2}) = f_r(x_1)\cdots f_0(x_{r+1})g_0(x_{r+2})\cdots g_k(x_{r+k+2}).
\end{align*}
Set
$
c_F = c\Vert f_u \Vert_{\infty}^r\Vert g_v \Vert_{\infty}^{k}\max\{\Vert f_u \Vert_{\infty}, \Vert g_v \Vert_{\infty} \}.
$
Since $f_i\in \cH_+(c,\vartheta)$ and $g_i\in \cH_-(c,\vartheta)$, we have
\beqn
x_j\mapsto F(x_1,\dots,x_{r+k+2})
\in
\begin{cases}
\cH_+(c_F,\vartheta), & 1\le j \le r+1,
\\
\cH_-(c_F,\vartheta), & r+2 \le j\le r+k+2.
\end{cases}
\eeqn
Hence, \(F\) is \((I_1,I_2)\)-admissible with the same parameters $c_F$ and \(\vartheta\). By Theorem~\ref{thm:multiple},
\begin{align*}
&\left | \int \tilde f \cdot  \tilde g\circ T^{n}\,\rd\mu - \int \tilde f\,\rd\mu \int\tilde g\,\rd\mu \right | 
\le M\!\left(\frac{c_F}{1-\vartheta} + \Vert F \Vert_{\infty} \right)\!\theta^n ,
\end{align*}
which implies the desired bound.
\end{proof}

\subsection{Multivariate normal approximation by Stein's method}
In this section and the next, we show that Theorem~\ref{thm:multiple} implies not only normal convergence but also estimates on the speed of convergence. In particular, we treat the case of multivariate normal distributions arising from vector-valued observables.

Let $T:X\to X$ be a general transformation preserving a probability measure~$\mu$. We introduce the following notations: Given an observable $f:X\to\bR^d$, we write
\beqn
f^k = f\circ T^k
\eeqn
for all $k\ge 0$, denoting the coordinate functions of $f^k$ by $f_\alpha^k$, $\alpha\in\{1,\dots,d\}$. We set
\beqn
W_N = \frac{1}{\sqrt{N}}\sum_{k=0}^{N-1} f^k
\eeqn
for all $N\ge 0$. For $0\le K < N$, we introduce the time window
\beqn
[n]_{N,K} = \{ k \in \bN_0 \cap [0,N-1] \, : \, |k-n| \le K \}
\eeqn
around $n\ge 0$, and define 
\beqn
W^n_{N,K} = W_N - \frac{1}{\sqrt{N}} \sum_{k\in [n]_{N,K}} f^k
\eeqn
for $0 \le n< N-1$. Thus, $W^n_{N,K}$ is a modification of $W_N$ where the times $k\in [n]_{N,K}$ are omitted in the sum. Finally, $\Phi_\Sigma(h)$ stands for the expectation of a function~$h:\bR^d\to\bR$ with respect to the centered multivariate normal distribution with covariance matrix~$\Sigma$. We write $\|f\|_\infty = \max_{1\le\alpha\le d}\|f_\alpha\|_d$, $\|D^2 h\|_\infty = \max_{1\le\alpha,\beta\le d}\|\partial_\alpha\partial_\beta h\|_\infty$, etc., for the norms of tensor fields.

The following theorem was proved in~\cite{HellaLeppanenStenlund_2016}, where an adaptation of Stein's method~\cite{Stein_1972} to the study of dynamical systems was developed:

\begin{thm}\label{thm:Stein} Let \(f: \, X \to \bR^d\) be a bounded measurable function with \(\mu(f) = 0\). Suppose \(h: \, \bR^d \to \bR\) is three times differentiable with \(\Vert D^k h \Vert_{\infty} < \infty\) for \(1 \le k \le 3\). Fix integer $N>2$. Suppose there exists $\theta\in(0,1)$ such that the following conditions are satisfied:

\begin{itemize}
\item[\bf (A1)] There exist constants $C_2 > 0$ and $C_4 > 0$ such that
\begin{align*}
|\mu(f_{\alpha} f^k_{\beta})| & \le C_2\,\theta^k
\\
|\mu(f_{\alpha} f_{\beta}^l f_{\gamma}^m f_{\delta}^n)| & \le C_4 \min\{\theta^l,\theta^{n-m}\}
\\
|\mu(f_{\alpha} f_{\beta}^l f_{\gamma}^m f_{\delta}^n) - \mu(f_{\alpha} f_{\beta}^l)\mu(f_{\gamma}^m f_{\delta}^n)| & \le C_4\,\theta^{m-l}
\end{align*}
hold whenever $k\ge 0$; $0\le l\le m \le n < N$; and $\alpha,\beta,\gamma,\delta\in\{1,\dots,d\}$.

\smallskip
\item[\bf (A2)] There exists a constant $C_0$ such that
\begin{align*}
|\mu( f^n \cdot  \nabla h(v + W^n_{N,K} t) ) | \le C_0\,\theta^K
\end{align*}
holds for all $0\le n < N$, $0\le t\le 1$, $v\in\bR^d$, and $0\le K<N$.

\smallskip
\item[\bf (A3)] $f$ is not a coboundary in any direction.\footnote{This is a standard condition, requiring that, given $u\in\bR^d\setminus\{0\}$, the scalar function $u\cdot f$ cannot be written in the form~$g-g\circ T$ for any $L^2(\mu)$ function $g:X\to\bR$.}

\end{itemize}
Then
\beq\label{eq:Sigma}
\Sigma = \mu(f \otimes f) + \sum_{n=1}^{\infty} (\mu(f^n \otimes f) + \mu(f \otimes f^n))
\eeq
is a well-defined, symmetric, positive-definite, $d\times d$ matrix; and
\beq\label{eq:Stein_billiars}
|\mu(h(W_N)) - \Phi_{\Sigma}(h)| 
\le \left(C_* \! \left(\tfrac{2}{|{\log\theta}|} + \tfrac{\theta}{\sqrt 3(1-\theta)}\right) + C_0\right)\cdot \frac{\log N}{\sqrt N}
\eeq 
where
\beqn
C_* = 12d^3\max\{C_2,\sqrt{C_4}\}\left(\Vert D^2 h\Vert_{\infty} + \Vert f \Vert_{\infty} \Vert D^3 h \Vert_{\infty} \right)\sum_{i = 0}^\infty (i+1)\theta^i .
\eeqn
\end{thm}

Returning to billiards, we now prove the following:
\begin{thm}\label{thm:Stein_billiards}
Assume $f:X\to\bR^d$ is bounded, $\int f\,\rd\mu = 0$, and there exist constants~$c\ge 0$ and~$\vartheta\in(0,1)$ such that 
$
f_\alpha \in\cH_-(c,\vartheta)\cap\cH_+(c,\vartheta)
$
for all $\alpha\in\{1,\dots,d\}$.
Then, for all $N$, condition (A1) is satisfied with
\begin{align*}
C_2 = M\Vert f \Vert_{\infty}\!\left(\frac{c}{1-\vartheta} +  \|f\|_{\infty}\right)
\quad\text{and}\quad
C_4 = M\Vert f \Vert_{\infty}^3\!\left(\frac{c}{1-\vartheta} +  \|f\|_{\infty}\right),
\end{align*}
and condition (A2) is satisfied with
\begin{align*}
C_0 = M\!\left(d^2c\frac{ \Vert \nabla h \Vert_{\infty}  + \Vert f \Vert_{\infty} \Vert D^2 h \Vert_{\infty} }{1-\vartheta} + d\Vert f \Vert_{\infty} \Vert \nabla h \Vert_{\infty}\right).
\end{align*}
\end{thm}

A similar result (with different constants) was recently proved in~\cite{HellaLeppanenStenlund_2016}, but there a direct scheme for checking~(A2) was implemented. Here we illustrate that~(A1) and~(A2) --- as well as the bound in \eqref{eq:Stein_billiars} --- are immediate consequences of Theorem~\ref{thm:multiple}.

\begin{proof}[Proof of Theorem~\ref{thm:Stein_billiards}]
That (A1) holds with the given expressions of $C_2$ and $C_4$ is immediate; see Theorem \ref{thm:multicorr_bound}. Condition (A2) follows by applying Theorem \ref{thm:multiple} to the function
\begin{align*}
F(x_0,\ldots, x_{n-K},x_n, x_{n+K}, \ldots, x_{N-1}) = f(x_n) \cdot \nabla h \!\left(v + \frac{1}{\sqrt{N}} \sum_{k \in [0,N)\setminus [n]_{N,K}} f(x_k)\,t \right),
\end{align*}
and two index sets \(I_1\le I_2\), where either \(I_1 = (0,\ldots,{n-K},n)\) and \(I_2 = (n+K,\ldots,N-1)\) (case \(n \ge K\)), or \(I_1 = (0,\ldots,{n-K})\) and \(I_2 = (n,n+K,\ldots,N-1)\) (case \(n < K\)). The function \(x_n \mapsto F(x_0,\ldots, x_{N-1})\) belongs to
\begin{align*}
\cH_{-}(d \Vert \nabla h \Vert_{\infty} c, \vartheta) \cap \cH_{+}(d \Vert \nabla h \Vert_{\infty} c, \vartheta),
\end{align*}
and for other indices \(r \neq n\), the function \(x_r \mapsto F(x_0,\ldots, x_{N-1})\) belongs to
\begin{align*}
\cH_{-}(N^{-\frac12}d^2\Vert f \Vert_{\infty} \Vert D^2 h \Vert_{\infty} c, \vartheta) \cap \cH_{+}(N^{-\frac12}d^2\Vert f \Vert_{\infty} \Vert D^2 h \Vert_{\infty} c, \vartheta).
\end{align*}
Hence, we see that \(F\) is \((I_1,I_2)\)-admissible with the same parameters
\begin{align*}
d^2c( \Vert \nabla h \Vert_{\infty}  + \Vert f \Vert_{\infty} \Vert D^2 h \Vert_{\infty} )
\quad\text{and}\quad \vartheta.
\end{align*}
By Theorem \ref{thm:multiple}, (A2) is satisfied with the value of $C_0$ given.
\end{proof}

\subsection{Multivariate normal approximation by P\`ene's method} In~\cite{Pene_2005}, P\`ene introduced a method of multivariate normal approximation based on the work of Rio~\cite{Rio_1996}; see also~\cite{Pene_2002,Pene_2004} for earlier, related, results by the same author. The theorem below is a special case of P\`ene's theorem applied to a map $T:X\to X$ preserving a probability measure~$\mu$. We write
\beqn
S_N = \sum_{k = 0}^{N-1} f^k.
\eeqn
Otherwise the notation is the same as in the previous section.

\begin{thm}\label{thm:Pene} Let $f:X\to\bR^d$ be a bounded measurable function with $\mu(f) = 0$. Suppose that there exist \(r \in \bZ_+\), \(C \ge 1\), \(M \ge \max\{ 1, \Vert f \Vert_{\infty} \}\) and a sequence of non-negative real numbers \( (\varphi_{p,l})_{p,l}\) such that the following conditions hold:
\begin{itemize}
\item[\bf (B1)] \(\varphi_{p,l} \le 1 \) and \( \sum_{p=1}^{\infty} p \max_{0 \le l \le \round{p/(r+1)}} \varphi_{p,l} < \infty\). \\

\item[\bf (B2)] For any integers \(a,b,c\) satisfying \(1 \le a + b + c \le 3\); for any integers \(i,j,k,p,q,l\) with \(0 \le i \le j \le k \le k + p \le k + p + q \le k + p + l\); for any \(\alpha, \beta, \gamma \in \{1,\ldots , d\}\); and for any bounded differentiable function \(G : \, \bR^d \times ([-M,M]^d)^3 \to \bR\) with bounded gradient,
\begin{align*}
&| {\operatorname{Cov_\mu}}[ G(S_{i},f^i,f^j,f^k),\, (f^{k+p}_{\alpha})^a(f^{k+p+q}_{\beta})^b(f^{k+p+l}_{\gamma})^c\,] | \le C(\Vert G \Vert_{\infty} + \Vert \nabla G \Vert_{\infty})\, \varphi_{p,l}.
\end{align*}
\end{itemize}
Then the limit
\beq\label{eq:Sigma_limit}
\Sigma = \lim_{N \to \infty} \mu(W_N \otimes W_N)
\eeq
exists. If \(\Sigma = 0\), then the sequence $(S_N)_{N\ge 0}$ is bounded in \(L^2(\mu)\). Otherwise there exists \(B > 0\) such that for any Lipschitz continuous function \(h: \, \bR^d \to \bR\),
\begin{align*}
| \mu(h(W_N)) - \Phi_{\Sigma}(h) | \le \frac{B\,\textnormal{Lip}(h)}{\sqrt{N}}
\end{align*}
for all $N\ge 1$.
\end{thm}

We proceed to the case of billiards:

\begin{thm}\label{thm:Pene_billiards}
Assume $f:X\to\bR^d$ is bounded, $\int f\,\rd\mu = 0$, and there exist constants~$c\ge 0$ and~$\vartheta\in(0,1)$ such that 
$
f_\alpha \in\cH_-(c,\vartheta)\cap\cH_+(c,\vartheta)
$
for all $\alpha\in\{1,\dots,d\}$. Then conditions (B1) and (B2) are satisfied with $\varphi_{p,l} = \theta^p$ and
\begin{align*}
C = M\frac{6d(c+1)\max\{1,\Vert f \Vert_{\infty}^5 \}}{1-\vartheta}.
\end{align*}
\end{thm}

The result is due to P\`ene \cite{Pene_2005}, assuming piecewise H\"older continuous observables. The above version covers also dynamically H\"older continuous observables. But again, our intention here is to underline that the conditions of P\`ene's theorem are immediate consequences of Theorem~\ref{thm:multiple}.

\begin{proof}[Proof of Theorem~\ref{thm:Pene_billiards}]
Obviously $\varphi_{p,l} = \theta^p$ satisfies (B1). To establish condition (B2), define
\begin{align*}
&F(x_0,\ldots, x_{i-1}, x_i, x_j, x_k, x_{k+p}, x_{k+p+q}, x_{k+p+l}) \\
&= G\!\left(\sum_{m=0}^{i-1} f(x_m), f(x_i), f(x_j),f(x_k)\right)f_{\alpha}^a(x_{k+p})f_{\beta}^b(x_{k+p+q})f_{\gamma}^c(x_{k+p+l}).
\end{align*}
Then \(\Vert F \Vert_{\infty} \le \Vert G \Vert_{\infty} \Vert f \Vert_{\infty}^{a+b+c}\), and \(F\) is  \((I_1,I_2)\)-admissible, where \(I_1 = (0,\ldots,i-1,i,j,k)\) and \(I_2 = (k+p,k+p+q,k+p+l)\): For the indices \(r \in \{k+p,k+p+q,k+p+l\} \), since \(a,b,c\) are integers with \(1 \le a+b+c \le 3\), a simple computation shows that the function \(x_r \mapsto F(x_0,\ldots, x_{k+p+l})\) belongs to
\begin{align*}
\cH_{-}(\Vert G \Vert_{\infty}3\max\{1,\| f \|_{\infty}^5\}c, \vartheta) \cap \cH_{+}(\Vert G \Vert_{\infty}3\max\{1,\| f \|_{\infty}^5\}c, \vartheta).
\end{align*}
Moreover, for \(r \in \{0,\ldots,i-1,i,j,k\}\), the function \(x_r \mapsto F(x_0,\ldots, x_{k+p+l})\) is in
\begin{align*}
\cH_{-}(d\Vert \nabla G \Vert_{\infty}\max\{1,\| f \|_{\infty}^3\}c, \vartheta) \cap \cH_{+}(d\Vert \nabla G \Vert_{\infty}\max\{1,\| f \|_{\infty}^3\}c, \vartheta).
\end{align*}
Consequently, \(F\) is \((I_1,I_2)\)-admissible with the same parameters
\begin{align*}
3dc\max\{1,\Vert f \Vert_{\infty}^5 \}(\Vert G \Vert_{\infty} + \Vert \nabla G \Vert_{\infty})
\quad\text{and}\quad
\vartheta.
\end{align*}
Theorem \ref{thm:multiple} applied to \(F\) and \((I_1,I_2)\) now yields the estimate
\begin{align*}
&| {\operatorname{Cov_\mu}}[ G(S_{i},f^i,f^j,f^k),\, (f^{k+p}_{\alpha})^a(f^{k+p+q}_{\beta})^b(f^{k+p+l}_{\gamma})^c\,] | \\
&\le M\!\left(\frac{3dc\max\{1,\Vert f \Vert_{\infty}^5 \}(\Vert G \Vert_{\infty} + \Vert \nabla G \Vert_{\infty})}{1-\vartheta} + \Vert F \Vert_{\infty}\right)\!\theta^{p} \\
&\le M\!\left(\frac{3dc\max\{1,\Vert f \Vert_{\infty}^5 \}(\Vert G \Vert_{\infty} + \Vert \nabla G \Vert_{\infty})}{1-\vartheta} + \Vert G \Vert_{\infty} \Vert f \Vert_{\infty}^{a+b+c}\right)\!\theta^{p} \\
&\le M\frac{6d(c+1)\max\{1,\Vert f \Vert_{\infty}^5 \}}{1-\vartheta}(\Vert G \Vert_{\infty} + \Vert \nabla G \Vert_{\infty})\theta^p.
\end{align*}
Hence, (B2) holds with the value of~$C$ given.
\end{proof}

\subsection{Vector-valued almost sure invariance principle by Gou\"ezel's method}
In this section we present an application of Theorem~\ref{thm:multiple} to multivariate almost sure limits.

The following theorem is due to Gou\"ezel \cite{Gouezel_2010}.
\begin{thm}\label{thm:Gouezel}
Let \(f: \, X \to \bR^d\) be a bounded measurable function with \(\mu(f) = 0\). Given integers $n>0$, $m>0$, $0\leq b_1<b_2<\dots<b_{n+m+1}$, $k\geq 0$, and vectors $t_1,\dots,t_{n+m}\in\bR^d$, set
\beqn
X_{n,m}^{(k)} = \sum_{j=n}^m t_j\cdot\sum_{i=b_j+k}^{b_{j+1}-1+k}f^i
\eeqn
for brevity.  Now, suppose there exist constants $t>0$, $C>0$, $C'>0$ and $\theta\in(0,1)$ such that
\beq\label{eq:Gouezel}
\begin{split}
& \left|\mu\!\left(e^{\imag X_{1,n}^{(0)}+\imag X_{n+1,n+m}^{(k)}}\right)-\mu\!\left(e^{\imag X_{1,n}^{(0)}}\right)\! \mu\!\left(e^{\imag X_{n+1,n+m}^{(k)}}\right)\right|
\le C\theta^k\!\left(1+\max_{1\leq j\leq n+m}|b_{j+1}-b_j|\right)^{C'(n+m)}
\end{split}
\eeq
holds for all choices of the numbers $n$, $m$, $b_j$, $k>0$, and all vectors $t_j$ satisfying $|t_j|<t$.
Then 
\begin{enumerate}
\item 
Equation \eqref{eq:Sigma}
yields a well-defined, symmetric, positive-semidefinite, $d\times d$ matrix~$\Sigma$.
\medskip
\item The matrix $\Sigma$ satisfies~\eqref{eq:Sigma_limit}.
\medskip
\item $W_N$ converges in distribution to~$\cN(0,\Sigma)$.
\medskip
\item Given any $\lambda>\frac14$, there exists a probability space together with two $\bR^d$-valued processes $(Y_n)_{n\geq 0}$ and $(Z_n)_{n\geq 0}$ such that
\medskip
\begin{enumerate}
\item $(f^n)_{n\geq 0}$ and $(Y_n)_{n\geq 0}$ have the same distribution.
\medskip
\item The random vectors $Z_n\sim\cN(0,\Sigma)$ are independent.
\medskip
\item Almost surely, $|\sum_{k=0}^{n-1} Y_k - \sum_{k=0}^{n-1} Z_k| = o(n^\lambda)$.  
\end{enumerate}
\end{enumerate}
\end{thm}

Such a theorem has a multitude of interesting consequences, including the central limit theorem (CLT), weak invariance principle, almost sure CLT, law of the iterated logaritm (LIL), Strassen's functional LIL, an upper and lower class refinement of the LIL, and an upper and lower class refinement of Chung’s LIL. We refer the reader to~\cite{Strassen_1964,Billingsley_convergence,PhilippStout_1975,LaceyPhilipp_1990,MelbourneNicol_2009} for more details concerning the implications.

We proceed to check condition~\eqref{eq:Gouezel} in the case of billiards. This was done by direct means in~\cite{Stenlund_2012}. To our knowledge, the resulting vector-valued almost sure invariance principle comes with the smallest error and covers the broadest class of observables to date. Here we show condition~\eqref{eq:Gouezel} to be an immediate consequence of Theorem~\ref{thm:multiple}.

\begin{thm}
Assume $f:X\to\bR^d$ is bounded, $\int f\,\rd\mu = 0$, and there exist constants~$c\ge 0$ and~$\vartheta\in(0,1)$ such that 
$
f_\alpha \in\cH_-(c,\vartheta)\cap\cH_+(c,\vartheta)
$
for all $\alpha\in\{1,\dots,d\}$. Given $t>0$,
\begin{align*}
& \left|\mu\!\left(e^{\imag X_{1,n}^{(0)}+\imag X_{n+1,n+m}^{(k)}}\right)-\mu\!\left(e^{\imag X_{1,n}^{(0)}}\right)\! \mu\!\left(e^{\imag X_{n+1,n+m}^{(k)}}\right)\right| 
\le \sqrt{2}M\!\left(\frac{t\sqrt{d}c}{1-\vartheta} + 1 \right)\!\theta^{k}
\end{align*}
holds for all choices of the numbers $n$, $m$, $b_j$, $k>0$, and all vectors $t_j$ satisfying $|t_j|<t$.
\end{thm}

\begin{proof}
Let  \(I_1= ( b_j,\ldots,b_{j+1}-1 \: : \: 1 \le j \le n )\) and \(I_2= ( b_j+k,\ldots,b_{j+1}-1+k \: : \: n+1 \le j \le n+m )\).
Define the function
\begin{align*}
&F(x_{b_1},\ldots, x_{b_2-1},\ldots, x_{b_{n+m}+k},\ldots, x_{b_{n+m+1}-1+k} ) \\
&= \exp\!\left(\imag \sum_{j=1}^n t_j\cdot\sum_{i=b_j}^{b_{j+1}-1}f(x_i) + \imag \sum_{j=n+1}^{n+m} t_j\cdot\sum_{i=b_j+k}^{b_{j+1}-1+k}f(x_i)\right).
\end{align*}
Then \(F\) is $(I_1,I_2)$-admissible with the same parameters $t\sqrt{d}c$ and $\vartheta$. Indeed, for all indices \(r\),
\begin{align*}
x_r \mapsto F(x_{b_1},\ldots, x_{b_{n+m+1}-1+k}) \in \cH_{-}(t\sqrt{d}c, \vartheta) \cap \cH_{+}(t\sqrt{d}c, \vartheta).
\end{align*}
To see this, recall that $|e^{\imag a} - e^{\imag b}| \le  |a-b|$ for all $a,b\in\bR$.
Thus, if say \(r = b_1\) and \(x \in W^u(y)\),
\begin{align*}
&|F(x,x_{b_1+1},\ldots, x_{b_{n+m+1}-1+k}) - F(y,x_{b_1+1},\ldots, x_{b_{n+m+1}-1+k})| \\
&\le |t_1 \cdot f(x) - t_1 \cdot f(y)| 
\le |t_1| \left(\sum_{1\le\alpha\le d} |f_\alpha(x) - f_\alpha(y)|^2\right)^{1/2}
\le t\sqrt{d} c\,\vartheta^{s_+(x,y)}.
\end{align*}
The other indices and local stable manifolds are treated similarly. Theorem~\ref{thm:multiple} now yields
\begin{align*}
& \left|\mu\!\left(e^{\imag X_{1,n}^{(0)}+\imag X_{n+1,n+m}^{(k)}}\right)-\mu\!\left(e^{\imag X_{1,n}^{(0)}}\right)\! \mu\!\left(e^{\imag X_{n+1,n+m}^{(k)}}\right)\right| 
\le \sqrt{2}M\!\left(\frac{t\sqrt{d}c}{1-\vartheta} + \|F\|_\infty \right)\!\theta^{k}.
\end{align*}
Since $\|F\|_\infty = 1$, the proof is complete.
\end{proof}


\section{Proofs of Theorems~\ref{thm:pair} and~\ref{thm:multiple}}\label{sec:proofs}
We begin by recalling three facts from the theory of billiards, which are necessary for the proofs of the theorems. We refer the reader to the standard textbook~\cite{ChernovMarkarian_2006} for more details.

\begin{lem}\label{lem:measurable_partition}
The space $(X,{\rm Borel},\mu)$ is a standard probability space, and the family $\xi = \{\xi_q\,:\,q\in\cQ\}$ of local unstable manifolds is a measurable partition of it. Here~$\cQ$ is an uncountable index set. Thus, the measure $\mu$ admits a disintegration
\beqn
\mu = \int_\cQ \nu_q\,\rd\lambda(q),
\eeqn
where the $\{\nu_q\,:\,q\in\cQ\}$ is a system of conditional probability measures of~$\mu$ on~$\xi$, with $\nu_q(\xi_q) = 1$ almost surely, and $\lambda$ is a factor probability measure on~$\cQ$. 
\end{lem}

\begin{lem}\label{lem:billiards_pushforward}
There exist system constants $a_0>0$, $M_0>0$ and $\theta_0\in(0,1)$ such that the following holds.
Suppose $G\in\cH_-(c,\vartheta)$. Then,
\beqn
\left|\int_{\xi_q} G\circ T^n\,d\nu_q - \int G\,d\mu\right| \le M_0(c + \|G\|_\infty) \max(\theta_0,\vartheta)^{\frac12(n-a_0|{\log|\xi_q|}|)}
\eeqn
for all~$n\ge 0$ and~$q\in\cQ$. Here $|\xi_q|$ stands for the length of the local unstable manifold~$\xi_q$.
\end{lem}

\begin{lem}\label{lem:billiards_tail}
There exists a system constant $M_1>0$ such that
\beqn
\int_\cQ |\xi_q|^{-1} \,d\lambda(q) \le M_1.
\eeqn
Moreover,
\beqn
\lambda(\{q\in\cQ\,:\, |\xi_q| \le \ve\}) \le M_1\ve
\eeqn
for all $\ve>0$.
\end{lem}

\bigskip
\bigskip
Next, let us recall a simple lemma.
\begin{lem}
Let $F:X^p\to\bR$ be a function, with $p\ge 1$ arbitrary. Then the identity
\beq\label{eq:telescope}
F(x_1,\dots,x_p) - F(y_1,\dots,y_p) = \sum_{r=1}^p\, [F(x_{<r},x_r,y_{>r}) - F(x_{<r},y_r,y_{>r})]
\eeq
holds for all $(x_1,\dots,x_p),(y_1,\dots,y_p)\in X^p$. Here we have denoted $x_{<r} = (x_1,\dots,x_{r-1})$ and $y_{>r} = (y_{r+1},\dots,y_p)$, with the agreement that $x_{<1} = y_{>p} = \emptyset$.
\end{lem}

\begin{proof}
The claim is tautological for $p=1$. For $p>1$, the induction step
\beqn
\begin{split}
& F(x_1,\dots,x_p) - F(y_1,\dots,y_p)
\\ 
=\ & F(x_{<p},x_p) - F(x_{<p},y_p) + F(x_{<p},y_p) + F(y_{<p},y_p)
\\
=\ & F(x_{<p},x_p,y_{>p}) - F(x_{<p},y_p,y_{>p}) + \sum_{r=1}^{p-1}\, [F(x_{<r},x_r,y_{>r}) - F(x_{<r},y_r,y_{>r})]
\end{split}
\eeqn
proves the lemma.
\end{proof}

The next lemma is a reflection of the fact that $\cH_-$ and $\cH_+$ are dynamically closed, as mentioned earlier; see also~\cite{Stenlund_2010}.
\begin{lem}\label{lem:dynHolder}
Let $F:X^{p+q}\to\bR$ be dynamically H\"older continuous in the sense that
\beqn
x_r \mapsto F(x_1,\dots,x_{p+q}) \in
\begin{cases}
\cH_+(c_r,\vartheta_+), & 1\le r\le p,
\\
\cH_-(c_r,\vartheta_-), & p+1 \le r\le p+q.
\end{cases}
\eeqn
Let $i_1<\dots<i_p\le 0$ and $0\le j_1<\dots<j_q$. Then
\beqn
x\mapsto F(T^{i_1}x,\dots,T^{i_p}x,y_1,\dots,y_q) \in \cH_+\!\left(\sum_{r = 1}^p c_r \vartheta_+^{- i_r},\vartheta_+\right)
\eeqn
for all $(y_1,\dots,y_q)\in X^q$ and
\beqn
y\mapsto F(x_1,\dots,x_p,T^{j_1}y,\dots,T^{j_q}y) \in \cH_-\!\left(\,\sum_{r = p+1}^{p+q} c_r \vartheta_-^{j_r},\vartheta_-\right)
\eeqn
for all $(x_1,\dots,x_p)\in X^p$.
\end{lem}

\begin{proof}
Denote $G(x) = F(T^{i_1}x,\dots,T^{i_p}x,y_1,\dots,y_q)$. Let $x$ and $\bar x$ belong to the same local unstable manifold, $x\in W^u(\bar x)$. Recalling $i_r\le 0$, we have $T^{i_r}x\in W^u(T^{i_r}\bar x)$, and $s_+(T^{i_r}x,T^{i_r}\bar x) = s_+(x,\bar x) - i_r$. Thus, identity~\eqref{eq:telescope} yields
\beqn
|G(x) - G(\bar x)| \le \sum_{r = 1}^p c_r \vartheta_+^{s_+(T^{i_r}x,T^{i_r}\bar x)} = \left(\sum_{r = 1}^p c_r \vartheta_+^{- i_r}\right)\! \vartheta_+^{s_+(x,\bar x)},
\eeqn
proving the first claim. A corresponding result holds for the inverse map $T^{-1}$, which is equivalent to the second claim.
\end{proof}

We are now ready to prove the first theorem, concerning $K = 2$.
\begin{proof}[Proof of Theorem~\ref{thm:pair}]
\beqn
I_1 = (i_1,i_2\dots,i_{p_1})
\quad\text{and}\quad
I_2 = (i_{p_1+1},i_{p_1+2}\dots,i_{p_2}).
\eeqn
Since the billiard process is stationary, the finite-dimensional distributions are the same for the translated index sets
\beqn
I_1' = I_1 - m = (i_1',i_2'\dots,i_{p_1}')
\quad\text{and}\quad
I_2' = I_2 - m = (i_{p_1+1}',i_{p_1+2}'\dots,i_{p_2}'),
\eeqn
where
\beqn
i_r' = i_r - m,  \qquad 1 \le r\le p_2,
\eeqn
and $m\in\bN$ is a number to be determined later. For the moment it suffices to assume that $i_{p_1} < m \le i_{p_1+1}$, meaning $i_{p_1}' < 0\le i_{p_1+1}'$.

For brevity, define
\beqn
G(x,y) = F(T^{i_1'}x,\dots,T^{i_{p_1}'}x,T^{i_{p_1+1}' - i_{p_1+1}'}y,\dots,T^{i_{p_2}' - i_{p_1+1}'}y).
\eeqn
Of course, we then have
\beqn
G(y,T^{i_{p_1+1}'} y) = F(T^{i_1'}y,\dots,T^{i_{p_1}'}y,T^{i_{p_1+1}'}y,\dots,T^{i_{p_2}'}y),
\eeqn
and
\beqn
\begin{split}
\int F\,\rd\bfP_{I_1'\cup I_2'}  = \int_\cQ\int_{\xi_q} G(y,T^{i_{p_1+1}'} y) \,\rd\nu_q(y)\,\rd\lambda(q),
\end{split}
\eeqn
Since $i_1'<\dots < i_{p_1}' < 0$, Lemma~\ref{lem:dynHolder} implies that 
\beqn
x\mapsto G(x,T^{i_{p_1+1}'} y) \in \cH_+\!\left(\sum_{r = 1}^p c_r \vartheta_+^{- i_r'},\vartheta_+\right),
\eeqn
so
\beqn
\sup_{y\in\xi_q}\left|\int_{\xi_q} G(x,T^{i_{p_1+1}'} y)\,d\nu_q(x) - G(y,T^{i_{p_1+1}'} y)\right| \le \sum_{r = 1}^{p_1} c_r \vartheta_+^{- i_r'}.
\eeqn
Inserting this estimate into the identity above, we obtain
\beqn
\left| \int F\,\rd\bfP_{I_1'\cup I_2'} - \int_\cQ\int_{\xi_q} \int_{\xi_q} G(x,T^{i_{p_1+1}'} y)\,\rd\nu_q(y) \,\rd\nu_q(x)\,\rd\lambda(q) \right| \le \sum_{r = 1}^{p_1} c_r \vartheta_+^{- i_r'}
\eeqn
after an application of Fubini's theorem.

Let us denote
\beqn
\hat\cQ = \lambda(\{q\in\cQ\,:\, |\xi_q| \le e^{-\frac{i_{p_1+1}'}{3a_0}}\}), 
\eeqn
where $a_0$ is the constant appearing in Lemma~\ref{lem:billiards_pushforward}.
Note that, by Lemma~\ref{lem:billiards_tail},
\beqn
\mu(\hat\cQ) \le M_1e^{-\frac{i_{p_1+1}'}{3a_0}}.
\eeqn
Since $0 = i_{p_1+1}' - i_{p_1+1}'<\dots< i_{p_2}' - i_{p_1+1}'$, Lemma~\ref{lem:dynHolder} implies that 
\beqn
y\mapsto G(x,y)\in\cH_-\!\left(\,\sum_{r=p_1+1}^{p_2} c_r\vartheta_-^{i_r' - i_{p_1+1}'},\vartheta_-\right).
\eeqn
Therefore, by Lemma~\ref{lem:billiards_pushforward},
\beqn
\begin{split}
& \left| \int_{\xi_q} G(x,T^{i_{p_1+1}'} y)\,\rd\nu_q(y) - \int G(x,y)\,\rd\mu(y)\right| 
\\
\le\ & M_0\!\left(\,\sum_{r=p_1+1}^{p_2} c_r\vartheta_-^{i_r' - i_{p_1+1}'} + \|G\|_\infty\right)\! \max(\theta_0,\vartheta_-)^{\frac12( i_{p_1+1}' - \frac13  i_{p_1+1}')} 
\\
\le\ & M_0\!\left(\,\sum_{r=p_1+1}^{p_2} c_r\vartheta_-^{i_r' - i_{p_1+1}'} + \|F\|_\infty\right)\! \max(\theta_0,\vartheta_-)^{\frac13 i_{p_1+1}' } 
\equiv E
\end{split}
\eeqn
if $q\in\cQ\setminus\hat\cQ$. On the other hand, if $q\in\hat\cQ$, we have the trivial bound
\beqn
\left| \int_{\xi_q} G(x,T^{i_{p_1+1}'} y)\,\rd\nu_q(y) - \int G(x,y)\,\rd\mu(y)\right| \le 2\|G\|_\infty \le 2\|F\|_\infty.
\eeqn
This yields
\beqn
\begin{split}
& \left| \int_{\xi_q}\int_{\xi_q} G(x,T^{i_{p_1+1}'} y)\,\rd\nu_q(y)\,\rd\nu_q(x) - \int_{\xi_q}\int G(x,y)\,\rd\mu(y)\,\rd\nu_q(x)\right|
\\
\le\ & E\, 1_{\cQ\setminus\hat\cQ}(q) + 2\|F\|_\infty 1_{\hat\cQ}(q).
\end{split}
\eeqn
Integrating the expression inside the absolute value with respect to $\rd\lambda(q)$, we obtain
\beqn
\begin{split}
& \Biggl| \int_\cQ \int_{\xi_q}\int_{\xi_q} G(x,T^{i_{p_1+1}'} y)\,\rd\nu_q(y)\,\rd\nu_q(x)\,\rd\lambda(q) 
\\
&\qquad\qquad - \int_\cQ\int_{\xi_q}\int G(x,y)\,\rd\mu(y)\,\rd\nu_q(x)\,\rd\lambda(q)\Biggr| 
\\
\le\ & 
E\lambda(\cQ\setminus\hat\cQ) +  2\|F\|_\infty\lambda(\hat\cQ)\le E + 2\|F\|_\infty M_1e^{-\frac{i_{p_1+1}'}{3a_0}}.
\end{split}
\eeqn
Hence,
\beqn
\begin{split}
& \left| \int F\,\rd\bfP_{I_1'\cup I_2'} - \int_\cQ\int_{\xi_q} \int G(x,y)\,\rd\mu(y) \,\rd\nu_q(x)\,\rd\lambda(q) \right| 
\\
\le\ & \sum_{r = 1}^{p_1} c_r \vartheta_+^{- i_r'} + E + 2\|F\|_\infty M_1e^{-\frac{i_{p_1+1}'}{3a_0}}.
\end{split}
\eeqn
Recalling $\int_\cQ \nu_q\,\rd\lambda(q) = \mu$, Fubini's theorem gives
\beqn
\begin{split}
& \left| \int F\,\rd\bfP_{I_1'\cup I_2'} - \iint G(x,y)\,\rd\mu(x) \,\rd\mu(y) \right|
\\
\le\ & \sum_{r = 1}^{p_1} c_r \vartheta_+^{- i_r'} + M_0\!\left(\,\sum_{r=p_1+1}^{p_2} c_r\vartheta_-^{i_r' - i_{p_1+1}'} + \|F\|_\infty\right)\! \max(\theta_0,\vartheta_-)^{\frac {i_{p_1+1}'}3} 
\\
& \qquad + 2\|F\|_\infty M_1e^{-\frac{i_{p_1+1}'}{3a_0}},
\end{split}
\eeqn
where we note that
\beqn
\iint G(x,y)\,\rd\mu(x) \,\rd\mu(y) = \int F\,\rd(\bfP_{I_1'}\otimes\bfP_{I_2''})
\eeqn
with $I_2'' = I_2' - i_{p_1+1}' = (i_{p_1+1}'- i_{p_1+1}',\dots,i_{p_2}'-i_{p_1+1}')$. Stationarity guarantees that 
\beqn
\bfP_{I_1'\cup I_2'} = \bfP_{I_1\cup I_2},\quad \bfP_{I_1'} = \bfP_{I_1}
\quad\text{and}\quad \bfP_{I_2''} = \bfP_{I_2}.
\eeqn
Inserting $i_r' = i_r - m$ into the upper bound above, we see that
\beqn
\begin{split}
& \left| \int F\,\rd\bfP_{I_1\cup I_2} -  \int F\,\rd(\bfP_{I_1}\otimes\bfP_{I_2}) \right|
\\
\le\ & \sum_{r = 1}^{p_1} c_r \vartheta_+^{i_{p_1} - i_r}\vartheta_+^{m - i_{p_1}} + M_0\!\left(\,\sum_{r=p_1+1}^{p_2} c_r\vartheta_-^{i_r - i_{p_1+1}} + \|F\|_\infty\right)\! \max(\theta_0,\vartheta_-)^{\frac {i_{p_1+1} - m}3} 
\\
& \qquad+ 2\|F\|_\infty M_1e^{-\frac{i_{p_1+1} - m}{3a_0}},
\end{split}
\eeqn
whenever $i_{p_1} < m \le i_{p_1+1}$. Finally, let $m$ be the smallest integer $\ge\frac14(3i_{p_1}+i_{p_1+1})$. Then $\frac13(i_{p_1+1}-m) \ge \frac14 (i_{p_1+1}-i_{p_1})-\frac13$ and $m-i_{p_1}\ge \frac14 (i_{p_1+1}-i_{p_1})$. This yields the final estimate
\beqn
\begin{split}
& \left| \int F\,\rd\bfP_{I_1\cup I_2} -  \int F\,\rd(\bfP_{I_1}\otimes\bfP_{I_2}) \right|
\\
\le\ & \left(\sum_{r = 1}^{p_1} c_r \vartheta_+^{i_{p_1} - i_r}\right)\!\vartheta_+^{\frac14 (i_{p_1+1}-i_{p_1})} 
\\
& \qquad + M_0\!\left(\,\sum_{r=p_1+1}^{p_2} c_r\vartheta_-^{i_r - i_{p_1+1}} + \|F\|_\infty\right)\! \max(\theta_0,\vartheta_-)^{\frac14 (i_{p_1+1}-i_{p_1})-\frac13} 
\\
& \qquad+ 2\|F\|_\infty M_1e^{-\frac{1}{4a_0}(i_{p_1+1}-i_{p_1})+\frac1{3a_0}}.
\end{split}
\eeqn
Defining the system constant $\theta_1 = e^{-\frac{1}{a_0}}$, we arrive at the claimed bound.
\end{proof}

We proceed to the proof of the second theorem, concerning $K\ge 2$. 

\begin{proof}[Proof of Theorem~\ref{thm:multiple}]
The proof is based on induction with respect to~$K$. 

\medskip
\noindent\underline{Case $K=2$}: The assumption is that $F$ is $(I_1,I_2)$-admissible with the same parameters~$c,\vartheta$. Therefore, Theorem~\ref{thm:pair} yields
\beqn
\begin{split}
& \left| \int F\,\rd\bfP_{I_1\cup I_2} -  \int F\,\rd(\bfP_{I_1}\otimes\bfP_{I_2}) \right|
\\
\le\ & \left(c \sum_{r = 1}^{p_1} \vartheta^{i_{p_1} - i_r}\right)\!\vartheta^{\frac14 \ell_1} 
+ M_0\!\left(c\sum_{r=p_1+1}^{p_2} \vartheta^{i_r - i_{p_1+1}} + \|F\|_\infty\right)\! \max(\theta_0,\vartheta)^{\frac14 \ell_1-\frac13} 
\\
&\qquad + 2M_1\|F\|_\infty  \theta_1^{\frac14 \ell_1-\frac13}
\\
\le\ & \frac c{1-\vartheta}\vartheta^{\frac14 \ell_1} 
+ M_0\!\left(\frac c{1-\vartheta} + \|F\|_\infty\right)\!\theta_0^{-\frac13} \max(\theta_0,\vartheta)^{\frac14 \ell_1} 
 + 2M_1\|F\|_\infty  \theta_1^{\frac14 \ell_1-\frac13}
 \\
\le\ & \left(\frac c{1-\vartheta}
+ M_0\!\left(\frac c{1-\vartheta} + \|F\|_\infty\right)\!\theta_0^{-\frac13}  
 + 2M_1\|F\|_\infty \theta_1^{-\frac13} \right)\!  \max(\vartheta,\theta_0,\theta_1)^{\frac14\ell_1}.
\end{split}
\eeqn
Defining the system constants $M = \max(1+M_0\theta_0^{-\frac13}\, ,\,M_0\theta_0^{-\frac13}+2M_1\theta_1^{-\frac13})$ and $\theta = \max(\vartheta,\theta_0,\theta_1)^{\frac14}$, we obtain
\beqn
\left|\int F\,\rd\bfP_{I} - \int F\,\rd(\bfP_{I_1}\otimes\bfP_{I_2})\right| \le M\!\left(\frac{c}{1-\vartheta} + \|F\|_\infty \right)\!\theta^{\ell_1}
\eeqn
as claimed.

\medskip
\noindent\underline{Case $K>2$}: We are now assuming that $F$ is $(I_1,\dots,I_K)$-admissible with the same parameters~$c,\vartheta$. In particular, $F$ is then $(I_1\cup\dots\cup I_{K-1}, I_K)$-admissible. Hence, the preceding case implies
\beqn
\left|\int F\,\rd\bfP_{I_1\cup\dots\cup I_K} -  \int F\,\rd(\bfP_{I_1\cup\dots\cup I_{K-1}}\otimes\bfP_{I_K})\right| \le M\!\left(\frac{c}{1-\vartheta} + \|F\|_\infty \right)\!\theta^{\ell_{K-1}}.
\eeqn
Suppose that
\beqn
\left|\int G\,\rd\bfP_{I_1\cup\dots\cup I_{K-1}} - \int G\,\rd(\bfP_{I_1}\otimes\dots\otimes \bfP_{I_{K-1}})\right| \le M\!\left(\frac{c}{1-\vartheta} + \|G\|_\infty \right)\!\sum_{k=1}^{K-2}\theta^{\ell_k}
\eeqn
for all $(I_1,\dots,I_{K-1})$-admissible functions $G$ with the same parameters $c,\vartheta$. It now suffices to just observe that $F$ is $(I_1,\dots,I_{K-1})$-admissible in its first $p_{K-1}$ arguments. More precisely, given $\bfy \in X^{p_K-p_{K-1}}$, the function 
\beqn
G_{\bfy}:X^{p_{K-1}}\to\bR: G_{\bfy}(\bfx) = F(\bfx,\bfy)
\eeqn
is $(I_1,\dots,I_{K-1})$-admissible, bounded by $\|F\|_\infty$. Hence,
\beqn
\begin{split}
& \left|\int F\,\rd(\bfP_{I_1\cup\dots\cup I_{K-1}}\otimes\bfP_{I_K}) - \int F\,\rd(\bfP_{I_1}\otimes\dots\otimes \bfP_{I_{K}}) \right|
\\
=\ & \left|\int\!\left(\int G_\bfy\,\rd \bfP_{I_1\cup\dots\cup I_{K-1}} - \int G_\bfy\,\rd(\bfP_{I_1}\otimes\dots\otimes \bfP_{I_{K-1}})\right)\rd\bfP_{I_K}(\bfy) \right|
\\
\le\ & M\!\left(\frac{c}{1-\vartheta} + \|F\|_\infty \right)\!\sum_{k=1}^{K-2}\theta^{\ell_k}.
\end{split}
\eeqn
This finishes the proof.
\end{proof}




%
\bibliography{Marginals}{}
\bibliographystyle{plainurl}


\vspace*{\fill}

\end{document}